\documentclass[10pt]{amsart}
  
\usepackage[T1]{fontenc}
\usepackage[small,it]{caption}
\usepackage[english]{babel}
\usepackage{url}
\usepackage{fullpage}
\usepackage{graphicx}                         
\usepackage{amsmath}
\usepackage{amsthm}                      
\usepackage{amssymb} 
\usepackage{mathrsfs}
\usepackage{stmaryrd}
\usepackage{enumerate}
\usepackage{tikz}
\usepackage{tikz-cd}

\usepackage{euscript}
\renewcommand{\mathcal}{\EuScript}
\swapnumbers
\theoremstyle{plain}
\makeatletter
\def\swappedhead#1#2#3{%
	\thmnumber{\@upn{\the\thm@headfont#2\@ifnotempty{#1}{.~}}}%
	\thmname{#1}%
	\thmnote{ {\the\thm@notefont(#3)}}}
\makeatother

\usepackage{color}
\usepackage{hyperref}
\usepackage[noabbrev,capitalize]{cleveref}
\hypersetup{
	colorlinks,
	linkcolor={red!50!black},
	citecolor={blue!50!black},
	urlcolor={blue!80!black}
}
\usepackage{palatino}
\newcommand{\myol}[2][3]{{}\mkern#1mu\overline{\mkern-#1mu#2}}
\usepackage{mathtools}

\newtheorem{thm}{Theorem}[subsection]
\newtheorem{lem}[thm]{Lemma}

\newtheorem{prop}[thm]{Proposition}
\newtheorem{cor}[thm]{Corollary}

\theoremstyle{definition}
\newtheorem{ex}[thm]{Example}
\newtheorem{rem}[thm]{Remark}
\newtheorem{remark}[thm]{Remark}
\newtheorem{defn}[thm]{Definition}

\newtheorem{para}[thm]{}

\DeclareMathOperator{\Spec}{Spec}

\title{Factorization statistics and bug-eyed configuration spaces}

\author{Dan Petersen}
\email{dan.petersen@math.su.se}
\thanks{D.P. acknowledges support by ERC-2017-STG 759082 and a Wallenberg Academy Fellowship. }
\thanks{P.T. acknowledges support by NSF-Grant No. DMS-1903040.}

\author{Philip Tosteson}
\email{ptoste@math.uchicago.edu}

\renewcommand{\SS}{\mathrm S}

\newcommand{\Sep}{\mathrm{Sep}}
\newcommand{\Supp}{\mathrm{Supp}}
\newcommand{\F}{\mathcal F}
\renewcommand{\L}{\mathcal L}
\newcommand{\Sal}{\mathrm{Sal}}
\newcommand{\Strat}{\mathrm{Strat}}

\newcommand{\Pol}{\mathrm{Pol}}	
\newcommand{\Aut}{\mathrm{Aut}}
	\newcommand{\Frob}{\mathrm{Frob}}
\newcommand{\A}{\mathbb A}

\newcommand{\Z}{\mathbf Z}
\newcommand{\Q}{\mathbf Q}
\newcommand{\R}{\mathbf R}
\newcommand{\C}{\mathbf C}

\newcommand{\bbC}{\mathbf C}

\newcommand{\rS}{\mathrm S}

\newcommand{\V}{\mathbb V}
\newcommand{\Ch}{{\mathrm{Ch}}}

\newcommand{\Fq}{\mathbf F_q}
\newcommand{\Conf}{\mathrm{Conf}}
\newcommand{\PConf}{\mathrm{PConf}}
\begin{document}

\begin{abstract}
A recent theorem of Hyde proves that the factorizations statistics of a random polynomial over a finite field are governed by the action of the symmetric group on the configuration space of $n$ distinct ordered points in $\R^3$. Hyde asked whether this result could be explained geometrically. We give a  geometric proof of Hyde's theorem as an instance of the Grothendieck--Lefschetz trace formula applied to an interesting, highly nonseparated algebraic space. An advantage of our method is that it generalizes uniformly to any Weyl group. In the process we study certain non-Hausdorff models for complements of hyperplane arrangements, first introduced by Proudfoot.
\end{abstract}
	\maketitle

\section{Introduction}	

\subsection{What, why, and how: factorization statistics and configuration spaces}	
	
\begin{para} Let $\chi \colon \SS_n \to \C$ be a class function.  We consider $\chi$ as a function on the set of integer partitions of $n$. Let $\Pol_n(\Fq)$ denote the set of monic polynomials of degree $n$ over the finite field $\Fq$. If $f \in \Pol_n(\Fq)$ has the factorization $f = \prod_i f_i$ into irreducible factors over $\Fq$, then the degrees of the $f_i$ form a partition of $n$, and we let $\chi(f)$ denote the value of $\chi$ on the resulting partition. If $f$ is squarefree, we may also define $\chi(f)$ as follows: the Frobenius automorphism $\Frob_q$ acts by permuting the roots of $f$, the resulting permutation is well defined up to conjugation, and we define $\chi(f)$ as the value of $\chi$ on this permutation. \end{para}

\begin{para} We will be interested in sums of the form $\sum_{f \in \Pol_n(\Fq)} \chi(f)$. These sums often have number-theoretic significance: we may for example compute the average number of irreducible factors, the variance, the probability of having more linear factors than quadratic factors, and so on, of a random degree $n$ polynomial over $\Fq$.  In general,  we may compute the expected value of any \emph{factorization statistic}, by which we mean any quantity $\chi(f)$  that only depends on the degrees of the irreducible factors of $f$.  The subject of this paper is a recent theorem of Hyde \cite{hyde} that relates these statistics to point configurations in $\R^3$. \end{para}

\begin{para}To state the theorem we need the following notation. If $V$ is a representation of the finite group $G$, then we denote by $\operatorname{ch} V$ its character. We denote by $\langle -,-\rangle_G$ the standard inner product of class functions, for which the characters of irreducible complex representations form an orthonormal basis.
\end{para}

\begin{thm}[Hyde] \label{hyde-thm}Let $\PConf_n(\R^3)$ be the configuration space of $n$ distinct ordered points in $\R^3$. The sum of $\chi$ over $\Pol_n(\Fq)$ is given by
\[
\sum_{f \in \Pol_n(\Fq)} \chi(f) = q^n \sum_{i\geq 0} q^{-i} \Big\langle \chi, \operatorname{ch} H^{2i}(\PConf_n(\R^3),\Q) \Big\rangle_{\SS_n}. 
\]
\end{thm}
\begin{rem}The odd degree cohomology groups of $\PConf_n(\R^3)$ vanish, which explains why Hyde's theorem only involves the even cohomology. 
\end{rem}

\begin{para}
It is natural to ask
\begin{quote}
	\emph{What is the meaning of this formula?}
\end{quote} If you have some experience with \'etale cohomology, your first thought is perhaps also: 
\begin{quote}\emph{This \textbf{must} be an instance of the Grothendieck--Lefschetz trace formula}.\end{quote} But how? Why should the two spaces $\Pol_n$ and $\PConf_n(\R^3)$ have any kind of relationship with each other? And isn't $\PConf_n(\R^3)$ very far from being a complex algebraic variety? Hyde's method of proof does not answer these questions: he calculates explicitly a generating function for all factorization statistics, considered as an element of the ring of symmetric functions, and then compares with the character of $H^\ast(\PConf_n(\R^d),\Q)$, which is well understood.  
\end{para}

\begin{para}In this paper we give a fully geometric proof of Hyde's result, and explain why there should be a connection between factorization statistics and point configurations in $\R^3$. In a nutshell, we will consider a certain non-Hausdorff scheme $X_n$ over $\Z$ with an action of $\SS_n$, satisfying the following two properties: (i) The analytification $X_n(\bbC)$ has the $\SS_n$-equivariant weak homotopy type of $\PConf_n(\R^3)$. (ii) The stack quotient $[X_n/\SS_n]$  --- which happens to be an algebraic space --- can be understood as the badly non-Hausdorff space parameterizing degree $n$ monic polynomials $p(t)$, together with a choice of factorization of $p(t)$ into irreducibles.  In particular we have a bijection of sets $[X_n/\SS_n](\Fq) = \Pol_n(\Fq)$.  
 Applying the Grothendieck--Lefschetz trace formula with twisted coefficients to $[X_n/\SS_n]$ produces Hyde's formula.   \end{para}

\begin{para}A reader may object to the claim in the preceding paragraph that the space parametrizing a degree $n$ polynomial with a choice of decomposition into irreducible factors is not Hausdorff. Surely every polynomial has a \emph{unique} factorization into irreducibles? Surely the space of polynomials with a fixed such factorization is just the space of polynomials, i.e.\ affine $n$-space? But this is precisely the crux of the matter --- the space of polynomials with a factorization into irreducibles is \emph{not} affine $n$-space but rather something deceptively similar: it is a \emph{bug-eyed cover} of affine $n$-space. A bug-eyed cover is a type of geometric peculiarity that occurs in the world of algebraic spaces: it is a space which fails to be Hausdorff not because it has ``too many points'' but rather because it has ``too many tangent directions''. \end{para}

\begin{para}The nonclassical nature of the space $[X_n/\SS_n]$, and the relationship to configurations of points in $\R^3$, can be understood geometrically in terms of the theory of foliations. (However, we will not use this perspective in the body of the paper.) Consider a foliated manifold $(M,\mathscr F)$. There is a natural induced foliation on $\PConf_n(M)$, which descends to the quotient manifold $\Conf_n(M) \coloneqq \PConf_n(M)/\SS_n$. Now let us specialize to the product foliation of $\R^3$ whose leaves are the lines parallell to the $z$-axis. Then the leaf space of the induced foliation of $\Conf_n(\R^3)$ can be identified with the analytification of $[X_n/\SS_n]$. In general the leaf space of a foliation is most naturally thought of as a Lie groupoid, the \emph{holonomy groupoid}. In this case the holonomy groupoid has trivial isotropy groups, but is still not weakly equivalent to a manifold, not even a non-Hausdorff manifold. On the algebraic side this is reflected in the fact that the algebraic space $[X_n/\SS_n]$ is not locally separated, which implies that it does not admit a well defined analytification in the category of complex analytic spaces. But if we consider $[X_n/\SS_n]$ as a stack then we obtain a well defined analytification as an \'etale groupoid in analytic spaces, weakly equivalent to the leaf space of the above foliation of $\Conf_n(\R^3)$. 
\end{para}

\begin{para}The scheme $X_n$ is a special case of a more general construction: for any arrangement $A$ of hyperplanes in $\R^n$ there exists a scheme $X_A$ with the property that the induced codimension three subspace arrangement in $(\R^3)^n$ is a fiber bundle over $X_A(\C)$ with fiber $\R^n$ (and in particular they are weakly homotopy equivalent), and similarly the complement of the complexification of the arrangement in $\C^n$ is an $\R^n$-fiber bundle over the set of real points $X_A(\R)$. The scheme $X_A$ was introduced by Proudfoot \cite{proudfoot} as a non-Hausdorff model for the complement of a complexified real hyperplane arrangement, although Proudfoot only considered the real points of $X_A$, not its set of complex points. \label{proudfoot}\end{para}

\begin{para}From our geometric approach,  we will also see that Hyde's theorem is a special case of a formula valid for any Weyl group.  This more general formula specializes to Hyde's theorem when the Weyl group $W$ is $\SS_n$ with its action on $\A^n$ by permuting the coordinates. 
\end{para}

\begin{thm}\label{generalized hyde}Let $W \subset \mathrm{GL}(n, \Z)$ be a Weyl group acting by reflections on affine $n$-space. Let $\chi \colon W \to \C$ be a class function. Let $M_W$ be the complement of the associated codimension 3 subspace arrangement in $(\R^3)^n$. For any finite field $\Fq$ we have
\[ \sum_{x \in (\A^n/W)(\Fq)} \chi(x) = q^n \sum_{i\geq 0} q^{-i} \Big \langle \chi,\operatorname{ch} H^{2i}(M_W,\Q)\Big\rangle_{W}. \]
\end{thm}

\begin{para}Let us explain the notation $\chi(x)$ in the left hand side of \cref{generalized hyde}. If $x \in (\A^n/W)(\Fq)$,  choose an arbitrary element $y \in \A^n(\overline{\mathbf F}_q)$ in the fiber of $\A^n \to \A^n/W$ over $x$. Since Frobenius acts preserving the fiber over $x$, there exists an element $w \in W$ such that $w \cdot y = \mathrm{Frob}(y)$. But $w$ is not unique in general; the best we can say is that we get a well defined coset $w \cdot \mathrm{Stab}(y) \subseteq W$, where $\mathrm{Stab}(y)$ is the stabilizer of $y$. But a stabilizer of a point is precisely what is called a \emph{parabolic} subgroup of a reflection group, and it is a standard fact that any coset of a parabolic subgroup has a distinguished \emph{minimal representative}  (which depends on the choice of a system of positive roots). Up to conjugacy the minimal representative does not depend on the choice of $y$ or the choice of system of positive roots, and we define $\chi(x)$ to be the value of $\chi$ on this conjugacy class. \end{para}

\begin{para}\label{type B}For a general Weyl group we do not have an interpretation of \cref{generalized hyde} as concrete as the one in the Type A case in terms of factorizations of polynomials. However, in the Type B case there is one: we obtain a formula that relates the expected factorizations of monic \emph{even} polynomials of degree $2n$ over finite fields to the action of the hyperoctahedral group on the cohomology of the complement of the Type B subspace arrangement in $(\R^3)^n$. 
\end{para}

\subsection{Structure of the paper}

\begin{para}This paper consists of two parts of rather different flavor. 

In \cref{section hyde} we describe previous work, and how we were led to define the scheme $X_n$; we explain how Hyde's theorem follows by applying the Grothendieck--Lefschetz trace formula to $[X_n/\SS_n]$; and finally we explain how to generalize the scheme $X_n$ to a more general scheme $X_W$ for a Weyl group $W$, such that Grothendieck--Lefschetz applied to $[X_W/W]$ produces \cref{generalized hyde}. This part of the paper is rather leisurely written, with few or no proofs. In particular we will in the course of the arguments use some properties of the schemes $X_n$ and $X_W$ without careful justifications, in order to not disrupt the flow of the text. 

The careful justifications are supplied in \cref{section three}, where we study the schemes $X_n$ and $X_W$ in detail. As mentioned previously these schemes are special cases of a more general construction due to Proudfoot, which associates a scheme $X_A$ to a hyperplane arrangement $A$ in $\R^n$. We show that the complement of the induced codimension $3$ subspace arrangement in $(\R^3)^n$ is an $\R^n$-bundle over $X_A(\C)$, and that the complement of the complexification is similarly fibered over $X_A(\R)$; in the real case this is a theorem of Proudfoot. The combinatorial structure of the scheme $X_A$ is independently interesting, and closely tied to the Salvetti complex. Using a theorem of Delucchi and Lofano--Paolini, we construct an algebraic cell decomposition (also known as an affine paving) of $X_A$, which can be used to determine its cohomology as a Galois representation.
\end{para}

\begin{para}We are grateful to Dustin Clausen, Emanuele Delucchi, Trevor Hyde, Dmitri Panov and David Rydh for comments and suggestions. Will Sawin has informed us that he has obtained a related geometric approach to Hyde's theorem. We also thank Michael Falk for informing us of Proudfoot's work. 
\end{para}

\section{The path to Hyde's theorem} \label{section hyde}

\subsection{The case of squarefree polynomials}

\begin{para}It will be convenient to rewrite the right hand side of Hyde's theorem \ref{hyde-thm} as follows. Let $V_\chi$ be a rational representation of $\SS_n$ and $\chi = \operatorname{ch} V_\chi$. If we can compute $\sum_{f \in \Pol_n(\Fq)} \chi(f)$ for any $V_\chi$, then we can compute this sum for an arbitrary class function by linearity, so it suffices to prove this case of Hyde's theorem. If $Y$ is any topological space then the finite-sheeted covering $\PConf_n(Y) \to \Conf_n(Y) = \PConf_n(Y)/\SS_n$ is classified by a homomorphism $\pi_1(\Conf_n(Y)) \to \SS_n$. Via this homomorphism, we may think of $V_\chi$ as defining a local system $\V_\chi$ on $\Conf_n(Y)$ for any $Y$, and there is an isomorphism
$$ H^k(\Conf_n(Y),\V_\chi) \cong H^k(\PConf_n(Y),\Q) \otimes_{\SS_n} V_\chi $$
for all $k$. 
In these terms we may state Hyde's formula equivalently in the form \begin{eqnarray}
\sum_{f \in \Pol_n(\Fq)} \chi(f) = q^n \sum_{i\geq 0} q^{-i} \dim H^{2i}(\Conf_n(\R^3),\V_\chi). \label{unweighted}
\end{eqnarray}
In this form the connection to Grothendieck--Lefschetz will become more transparent. \end{para}

\begin{para}
Before trying to find a geometric interpretation of Hyde's formula \eqref{unweighted} it is instructive to first consider a similar but simpler case, namely the case where we restrict $f$ to be \emph{squarefree}. In this setting, the analogue of Hyde's theorem is
\begin{eqnarray}
\sum_{\substack{f \in \Pol_n(\Fq) \\ f \text{ squarefree}}} \chi(f) = q^n \sum_{i\geq 0} (-1)^i q^{-i} \dim H^{i}(\Conf_n(\C),\V_\chi), \label{squarefree}
\end{eqnarray}which \emph{is} an instance of the Grothendieck--Lefschetz trace formula in a more direct manner. The formula \eqref{squarefree} can be found in a paper of Church--Ellenberg--Farb \cite{churchellenbergfarb-asymptotics} (as they point out, closely related statements were in the literature previously, e.g.\ \cite{kisinlehrer}). In fact, Hyde was led to discover his formula by investigating what happens when the ``squarefree'' condition is dropped from the left hand side of \eqref{squarefree}. We give a brief account of the proof of \eqref{squarefree} here, and refer to \cite{churchellenbergfarb-asymptotics} for a more detailed account. We will use the Grothendieck--Lefschetz formula in the form
\begin{equation}
\sum_{x \in X(\Fq)}\operatorname{Tr}(\Phi_q \mid \mathcal F_{\overline x}) = q^{\dim X}\sum_i (-1)^i \operatorname{Tr}(\Phi_q \mid H^i_{\text{\emph{\'et}}}(X_{\overline{\mathbf F}_q},\mathcal F)),\label{GL1}
\end{equation} 
where $X$ is a smooth finite type scheme over $\Fq$, $\mathcal F$ is a locally constant $\Q_\ell$-sheaf on $X$ 
(where $\ell$ is a prime invertible in $\Fq$), $\mathcal F_{\overline x}$ is the stalk of $\mathcal F$ at a geometric point $\overline x$ over $x$, and $\Phi_q$ is the arithmetic Frobenius (whose eigenvalues are the inverses of those of the geometric Frobenius). 
\end{para}
\begin{para}\label{localsystem}
We apply Grothendieck--Lefschetz to the scheme $\Conf_n(\A^1)$ over $\Z$, which parametrizes squarefree monic degree $n$ polynomials. In the same way that the representation $V_\chi$ defined a local system $\V_\chi$ on $\Conf_n(Y)$ for any topological space $Y$ we obtain a locally constant $\ell$-adic sheaf $\V_{\chi,\ell}$ on $\Conf_n(\A^1)$. Now recall that when the polynomial $f$ is squarefree we may also define $\chi(f)$ as follows: the Frobenius automorphism acts by permuting the roots of $f$, the resulting permutation is well defined up to conjugation, and we define $\chi(f)$ as the value of $\chi$ on this conjugacy class. But the monodromy of the covering $\PConf_n(\A^1) \to \Conf_n(\A^1)$ is precisely given by the action of $\SS_n$ permuting the roots, which implies that $\chi(f) = \operatorname{Tr}(\Phi_q \mid (\V_{\chi,\ell})_{\overline f})$, and so the left hand sides of \eqref{squarefree} and \eqref{GL1} agree. 
\end{para}

\begin{para}
To identify the right hand sides of  \eqref{squarefree} and \eqref{GL1} we need to prove that $$H^i(\Conf_n(\C),\V_\chi) \otimes \Q_\ell \cong H^i_{\text{\emph{\'et}}}(\Conf_n(\A^1)_{\overline{\mathbf F}_q},\V_{\chi,\ell})$$ for all $i$, or equivalently that there is an $\SS_n$-equivariant isomorphism $$H^i(\PConf_n(\C),\Q_\ell) \cong H^i_{\text{\emph{\'et}}}(\PConf_n(\A^1)_{\overline{\mathbf F}_q},\Q_\ell);$$ moreover, we need to know that all eigenvalues of $\Phi_q$ acting on $H^i_{\text{\emph{\'et}}}(\PConf_n(\A^1)_{\overline{\mathbf F}_q},\Q_\ell)$ equal $q^{-i}$. The first part can be deduced in all characteristics from the Artin comparison theorem and the fact that $\PConf_n(\A^1)$ is the complement of a simple normal crossing divisor in a smooth proper scheme over $\Z$ (using the Fulton--MacPherson compactification \cite{fmcompactification}). However to understand the Galois action as well,  it is more efficient to argue that both Betti and \'etale cohomology of the complement of an arrangement of hyperplanes can be computed by an identical deletion-restriction scheme; the latter argument implies in addition that the \'etale cohomology of the complement of any hyperplane arrangement is pure Tate of weight $2i$ in degree $i$, as needed. This argument is independently due to Kim and Lehrer \cite{kimhyperplane,lehrerhyperplane}. 
\end{para}

\begin{para}As we mentioned in \S\ref{type B}, our more general formula \cref{generalized hyde} has an interpretation in the Type B case in terms of expected factorizations of even polynomials of degree $2n$ over $\Fq$. There is also a ``squarefree''  Type B version of the formula,  relates the expected factorizations of squarefree even polynomials of degree $2n$ over $\Fq$ and the action of the hyperoctahedral group on the cohomology of the complement of the Type B arrangement in $\C^n$. This formula is is due to Jim\'enez Rolland and Wilson \cite{jimenezrollandwilson}, and is proven using the Grothendieck--Lefschetz trace formula in much the same way as in the Type A case. 
\end{para}

\begin{para}
The upshot is that in the squarefree case we have a good geometric understanding of the factorization statistics in terms of $\ell$-adic cohomology, and we would like a similar geometric description of the  factorization statistics for general polynomials. 
\end{para}
\subsection{The plot thickens: weighted statistics}
\begin{para}\label{two approaches}
We saw that to compute the factorization statistics of squarefree polynomials  \eqref{squarefree},  one uses the cover $\PConf_n(\A^1) \to \Conf_n(\A^1)$.   So at first it may seem that if we wish to compute the factorization statistics for \emph{all} polynomials,  we need to replace the cover of the space of squarefree polynomials, $\PConf_n (\A^1) \to \Conf_n(\A^1)$, by the ramified cover of the space of all polynomials, $h\colon (\A^1)^n \to \A^n$.   Even though $h$ is not \'etale, we can still associate a sheaf on $\A^n$  to every irreducible representation $V_\chi$ of $\SS_n$, namely the sheaf
$$ (h_* \Q_\ell)_\chi \coloneqq \mathrm{Hom}_{\SS_n}(V_\chi,h_\ast \Q_\ell).$$
Here $V_\chi$ is considered as a constant sheaf of $\Q_\ell[\SS_n]$-modules on $\A^n$ (and should not be confused with the local system $\V_\chi$), so that $(h_* \Q_\ell)_\chi \otimes_\Q V_\chi$ is the $V_\chi$-isotypic component of the sheaf $h_* \Q_\ell$ with respect to its canonical $\SS_n$-action.  Over the locus of squarefree polynomials $(h_* \Q_\ell)_\chi$ restricts to the local system $\V_{\chi,\ell}$ of  \S\ref{localsystem}.  \label{isotypicsheaf}
\end{para}

\begin{para}
 This approach to computing factorization statistics was carried out by Gadish  \cite{gadish-trace},  resulting in the following identity
 \begin{equation}
\sum_{f \in \Pol_n(\Fq)} \frac{1}{\vert \Aut(f)\vert} \sum_{\sigma \in \Aut(f)} \chi(F_f \circ \sigma)   =   q^n \cdot  
\langle \chi, \mathbf 1 \rangle_{\SS_n}.  \label{weighted}\end{equation}
Let us explain the notation $\sum_{\sigma \in \Aut(f)} \chi(F_f \circ \sigma)$.   Suppose $f \in \Pol_n(\Fq)$ factors as $f = \prod f_i^{a_i}$, where all the $f_i$ are distinct and irreducible, and let us number the roots of $f$ as $z_1,\ldots,z_n$. Then we define $\Aut(f) := \prod_i \SS_{a_i}$, and the ordering of the roots identifies $\Aut(f)$ with a Young subgroup of $\SS_n$.  We let $F_f \in \SS_n$ denote the permutation representing the action of Frobenius on the roots, i.e.\ such that $\Frob_q(z_i) = z_{F_f(i)}$. The permutation $F_f$ is not uniquely determined unless $f$ is squarefree. Nevertheless the sum
$$\sum_{\sigma \in \Aut(f)} \chi(F_f \circ \sigma)$$
is well defined and independent of our choices. Note that if $f$ is squarefree then this sum is just $\chi(f)$.

\begin{para}\label{coh of affine space}
Note also that the right hand side of \eqref{weighted} can trivially be rewritten as 
$$ \langle \chi, \mathbf 1 \rangle_{\SS_n} =  \sum_{i\geq 0} (-1)^i \langle \chi,  \operatorname{ch} H^{i}_{\text{\emph{\'et}}}((\A^1)^n, \Q_\ell) \rangle_{\SS_n}$$
since the cohomology of $(\A^1)^n$ is just a single copy of the trivial  representation of $\SS_n$ in degree $0$. Thus, as expected, on the right hand side of \eqref{weighted} we see the cohomology of $(\A^1)^n$, together with its symmetric group action. 
\end{para}

\begin{para}
Now, clearly   the application of Grothendieck--Lefschetz to the map $(\A^1)^n \to \A^n$ did not result in the formula for the factorization statistics that we wanted.  Instead it gave us a different formula, for \emph{weighted} factorization statistics.  Here the quantity associated to a polynomial $f$ involves an averaging procedure, weighted by $|\Aut(f)|^{-1}$,  where $\Aut(f)$ is the stabilizer group of the action of $\SS_n$ on the fiber above $f$.   So the key problem in recovering Hyde's formula from the action of $\SS_n$ on $(\A^1)^n$ seems to be that $\SS_n$ does not act freely.
\end{para}

\begin{para}Indeed, the presence of the weight $ |\Aut(f)|^{-1}$ in Gadish's formula suggests that it may alternately be obtained by using the \emph{stack quotient} $[(\A^1)^n/\SS_n]$. 
 The Grothendieck--Lefschetz trace formula for a smooth algebraic stack $\mathcal X$ takes the form \cite[Corollary 6.4.10]{behrend-derived} 
\begin{equation}
\sum_{x \in \mathcal X(\Fq)/\simeq} \frac 1 {\vert \Aut(x)\vert} \operatorname{Tr}(\Phi_q \mid \mathcal F_{\overline x}) = q^{\dim \mathcal X}\sum_i (-1)^i \operatorname{Tr}(\Phi_q \mid H^i_{\text{\emph{\'et}}}(\mathcal X_{\overline{\mathbf F}_q},\mathcal F))\label{GL}
\end{equation}where $\Aut(x)$ denotes the $\Fq$-points of the isotropy group of $\mathcal X$ at $x$.   This formula, applied to the cohomology of $[(\A^1)^n /\SS_n]$,  with coefficients in the local system $\V_\chi$ induced by the  \'etale $\SS_n$-cover $(\A^1)^n \to [(\A^1)^n /\SS_n]$,  reproduces Gadish's formula \eqref{weighted}.  In fact, Gadish proved a general formula applicable to a ramified $G$-cover $h \colon X \to X/G$, where $G$ is any finite group, and it can be realized as the Grothendieck--Lefschetz trace formula applied to $[X/G]$.  The reason why the two approaches produce the same formula is that the pushforward of $\V_\chi$ along $[X/G] \to X/G$ is precisely the sheaf $(h_* \Q_l)_{\chi}$  that Gadish uses to obtain his formula (\S\ref{isotypicsheaf}).
\end{para}

\begin{para}
 All of the above suggests that the non-free action of $\SS_n$ on $(\A^1)^n$ is the culprit that causes weighted factorization statistics to appear, rather than ordinary factorization statistics,  when applying Grothendieck--Lefschetz to $(\A^1)^n \to \A^n$.   In other words, the isotropy of the stack $[(\A^1)^n/\SS_n]$  is the source of the weights $|\Aut(f)|^{-1}$, and to recover Hyde's formula geometrically we should somehow try to remove the stackiness of $[(\A^1)^n/\SS_n]$.  We could try to do this by replacing $(\A^1)^n$  by a space $X_n$ that looks somehow ``similar'' to affine space but which has a free action of $\rS_n$.    In particular, there should be a bijection of sets $$[X_n/\SS_n](\Fq) \stackrel \sim \to \Pol_n(\Fq).$$   This is in contrast to $[(\A^1)^n/\SS_n](\Fq)$, which is a groupoid with  $\pi_0 [(\A^1)^n/\SS_n](\Fq) = \Pol_n(\Fq)$ but with non-trivial stabilizers. Moreover, we saw that the right hand side of Gadish's formula corresponds to the cohomology of $(\A^1)^n$ (\S\ref{coh of affine space}); if we want the Grothendieck--Lefschetz trace formula to replicate the right hand side of Hyde's formula \eqref{unweighted} then the space $X_n$ should have the same cohomology groups as $\PConf_n(\R^3)$, and the cohomology groups $H^{2i}_{\mathrm{\emph{\'et}}}$ should be pure Tate of weight $2i$. 
\end{para}

\subsection{Bug-eyed configuration spaces}\label{subsec:bug-eyed}

\begin{para}The simplest way to modify $(\A^1)^n$ into a space with a free $\rS_n$-action is to add extra data at the points stabilized by $\rS_n$.   Accordingly, we propose the following construction. Let us define a nonseparated scheme $X_n$ as the space parametrizing $n$ not necessarily distinct ordered points on the line $\A^1$, together with an auxilliary total ordering on each of the sets of points which happen to coincide. For example, $X_2$ is the plane with doubled diagonal. In general, $X_n$ is obtained by gluing together $n!$ copies of affine $n$-space $\A^n$ along open subspaces given by complements of diagonal loci. (A more general construction will be carefully described in \S\ref{construction}.) There is a natural action of $\SS_n$ on $X_n$, which is now a \emph{free} action: the loci in $\A^n$ which have been ``duplicated'' are precisely those with nontrivial $\SS_n$-stabilizers. Thus the quotient stack $[X_n/\SS_n]$ has trivial isotropy groups and is now actually an algebraic space. The algebraic space $[X_n/\SS_n]$ is smooth and looks deceptively similar to affine $n$-space $\A^n$: there is a natural map $p \colon [X_n/\SS_n] \to \A^n$, such that for the natural stratification of $[X_n/\SS_n]$ with strata indexed by the partition lattice, the restriction of $p$ to each stratum is an isomorphism onto its image. Nevertheless $p$ is not an isomorphism --- for one, $[X_n/\SS_n]$ is not separated (in fact not even locally separated), since $X_n$ is not separated. This is a geometric phenomenon which is impossible for schemes but possible for algebraic spaces: informally speaking, schemes can only be nonseparated because they have ``too many points'', whereas algebraic spaces can be nonseparated also because they have ``too many tangent directions'' (cf.\ \cite[Example 1, p.\ 9]{knutson-algebraicspaces}). For example, $[X_2/\SS_2]$ is the plane with doubled tangent directions along the diagonal. In general $[X_n/\SS_n]$ is a ``bug-eyed cover'' of $\A^n$, in the terminology of Koll\'ar \cite{kollar-bugeyed}.  
\end{para}
\begin{para}\label{firstclaim}

Our first claim is that if we associate a locally constant $\ell$-adic sheaf $\V_{\chi,\ell}$ on $[X_n/\SS_n]$ as before, then the left-hand side of the Grothendieck--Lefschetz trace formula applied to $\mathcal X = [X_n/\SS_n]$ and $\mathcal F = \V_{\chi,\ell}$ becomes precisely the left hand side of \eqref{unweighted}. The argument is almost exactly the same as in the squarefree case: we observe that since $[X_n/\SS_n] \to \A^n$ is a bug-eyed cover, it follows that $[X_n/\SS_n](\Fq) \to \Pol_n(\Fq)$ is a bijection. Moreover, the monodromy of $X_n \to [X_n/\SS_n]$ over a polynomial $f$ is given by $\SS_n$ acting on the set of roots of $f$. This identifies Frobenius acting on the set of roots with a permutation whose cycle type is given by the degrees of the irreducible factors of $f$. This explains why factorization statistics are connected to the cohomology of $X_n$.  But how is $X_n$ related to $\PConf_n(\R^3)$?
\end{para}
\begin{para}
Our second claim (\cref{weakequivalence}) is that there exists an $\SS_n$-equivariant map $$f \colon \PConf_n(\R^3) \to X_n(\C)$$ which is a fiber bundle with fiber $\R^n$ and in particular a weak homotopy equivalence. Let us define this map $f$. A point in $\R^3$ has an $x$, $y$, and a $z$-coordinate; by forgetting the $z$-coordinate we obtain a point in $\R^2 \cong \C$. In this way we obtain from a configuration of $n$ distinct ordered points in $\R^3$ a configuration of $n$ points in $\C$, where the points may now coincide with each other. But note that if some subset of these points in $\C$ are equal, then their preimages in $\R^3$ must have distinct $z$-coordinates. It follows that the $z$-coordinate defines a total ordering on this set of coinciding points. In this way we obtain from each point of $\PConf_n(\R^3)$ a point of $X_n(\C)$. We remark that this construction is quite similar to the stratification of $\Conf_n(\C)$ considered by Fox--Neuwirth \cite{foxneuwirth} and Fuks \cite[Section 3]{fuks}. 
\end{para}
\begin{para}
The two claims combined together give a ``philosophical'' explanation of Hyde's theorem: the sum $\sum_{f \in \Pol_n(\Fq)} \chi(f)$ can be evaluated in terms of the $\SS_n$-equivariant \'etale cohomology of $X_n$, and over the complex numbers the space $X_n$ has the same cohomology as $\PConf_n(\R^3)$. If we want to actually re-prove Hyde's theorem by purely geometric considerations we need two further ingredients:\begin{itemize}
	\item By the Artin comparison theorem, and general constructibility and base change results, the Betti cohomology of $X_n(\C)$ is isomorphic to the \'etale cohomology of $X_{n,\overline{ \mathbf F}_q}$ away from finitely many characteristics. To recover Hyde's theorem in all characteristics we need to know that the sheaves $R^q\pi_\ast \Q_\ell$ are locally constant over $\operatorname{Spec}(\Z[1/\ell])$, where $\pi \colon X_n \to \operatorname{Spec}(\Z)$ is the structural morphism. 
	\item To apply Grothendieck--Lefschetz, we need to know the Galois action on $H^i(X_{n,\overline{ \mathbf F}_q},\Q_\ell)$; specifically, what is needed for Hyde's theorem is the statement that $H^i$ is pure Tate of weight $i$. 
\end{itemize}
Just as in the squarefree case we will resolve these issues by finding a direct geometric method to compute the cohomology of $X_n$ that works independently of the characteristic and which also keeps track of the Galois action. 
\end{para}

\subsection{A cellular decomposition} \label{subsec:cellular}
\begin{para}There are many examples of algebraic varieties $Z$ with the property that $H^i(Z)$ is pure Tate of weight $i$ (and then nonzero only for even $i$), such as  smooth Schubert varieties and smooth projective toric varieties.  The most hands-on way in which this can happen is when $Z$ is smooth and admits an \emph{algebraic cell decomposition}, a certain kind of decomposition into locally closed subschemes each of which is isomorphic to an affine space $\A^d$ (\cref{cell def}). If this holds then $H^{2i}$ is pure Tate of weight $2i$ for all $i$, and if $Z$ is $d$-dimensional then $\dim H^{2i}(Z)$ is given by the number of cells of dimension $d-i$ (\cref{cellprop}).
 \end{para}
\begin{para}

We now claim that the schemes $X_n$ admit such an algebraic cell decomposition. For example, $X_2$ is the plane with doubled diagonal, which can be decomposed as the union of a plane and a line:
$$ X_2 = \A^2 \cup \A^1,$$ where $\A^2$ is the locus $\{ x_1 \neq x_2\} \cup \{x_1 = x_2, 1 \leq 2\}$  and $\A^1$ is the line $\{x_1 = x_2, 2 \leq 1\}$.   
Similarly $X_3$ looks like $\A^3$ except the planes defined by $x=y$, $x=z$ and $y=z$ have all been doubled, and there are six copies of the line $x=y=z$; this gives rise to a decomposition
$$ X_3 = \A^3 \cup \A^2 \cup \A^2 \cup \A^2 \cup \A^1 \cup \A^1.$$
What will happen in general is indeed that each $X_n$ admits an algebraic cell decomposition (which is far from unique or canonical, as can be seen even for $X_2$). Although this statement might seem simple, there is in fact a nontrivial combinatorial theorem lurking in the background here, which is a theorem of Delucchi \cite{delucchi} and Lofano--Paolini \cite{lofanopaolini} stating that the set of chambers of an arrangement of real hyperplanes always admits a \emph{``valid order''}. Each choice of valid order gives rise to an algebraic cell decomposition of $X_n$, so it follows that the Betti numbers of $X_n$ in \'etale cohomology are independent of the characteristic of the base field, and the $i$th \'etale cohomology group of $X_n$ is pure Tate of weight $i$ in any characteristic. Moreover, \emph{even though the cell decomposition is not $\SS_n$-equivariant}, it follows that the decomposition of the \'etale cohomology into irreducible $\SS_n$-representations is the same in any characteristic. Indeed, the $\ell$-adic sheaf on $\mathrm{Spec}(\Z[1/\ell])$ given by the the $q$th \'etale cohomology of $X_n$ with $\Q_\ell$-coefficients is the direct sum of its isotypical components under the action of $\SS_n$, by Maschke's theorem. But this sheaf is locally constant, and a direct summand of a locally constant sheaf is locally constant, so each isotypical component must itself be locally constant over $\mathrm{Spec}(\Z[1/\ell])$.   Hence the \'etale cohomology of $[X_n/\SS_n]$ with coefficients in $V_\chi$, which is identified with this isotypic component,  is locally constant as well.
\end{para}
\subsection{Conclusion}

\begin{para}We are now in a position to assemble what has been said so far into a proof. More precisely, 
putting these three ingredients together gives a completely geometric proof of Hyde's theorem:
\begin{itemize}
\item The identification between the Grothendieck--Lefschetz trace formula for $[X_n/\SS_n]$ and the left hand side of Hyde's formula (\S\ref{firstclaim}).
\item The weak equivalence between $\PConf_n(\R^3)$ and $X_n(\C)$ (\cref{weakequivalence}).
\item The space $X_n$ having an algebraic cell decomposition over the integers, so that its cohomology is pure of Tate type, with Betti numbers independent of characteristic of the base field (\cref{exists cell decomposition} and \cref{cellprop}).
\end{itemize}
\end{para}

\begin{rem}
	An amusing remark is that while $|\PConf_n (\A^1)(\Fq)| = q \cdot (q - 1) \cdots (q - (n - 1))$  has its point counts given by a falling Pochhammer symbol $(q)_{\underline n}$,  the point counts of  $X_n(\Fq)$ are given by a \emph{rising} Pochhammer symbol $(q)^{\underline n} = q (q+1)\cdots (q + n - 1)$.  We can see this directly by considering the map $X_n(\Fq) \to X_{n-1}(\Fq)$ which forgets the last point, and has fiber of size $q + (n - 1)$.  This may be compared with the map of configuration spaces $\PConf_n (\A^1) \to \PConf_{n-1} (\A^1)$, which can be used to derive in a similar fashion the number of $\Fq$-points of $\PConf_n (\A^1)$.      
 Furthermore, just as we may consider the configuration space $\PConf_n (Y)$ for any space $Y$,  we may  construct a variant of $X_n$ where  $\A^1$ is replaced by an arbitrary $Y$.  If we take $Y$ to be the finite set $[k] = \{1, \dots, k\}$,  then we obtain a set with $(k)^{\underline n}$  elements, on which $\SS_n$ acts freely, with quotient exactly the set of multisubsets of $[k]$.  In some sense, this whole paper is an elaboration on this bijective proof that $[k]$ has $\frac{(k)^{\underline n}}{n!}$ multisubsets.
\end{rem}

\begin{rem}\label{locally separated}
As a final remark on the proof let us point out that in our argument we applied Artin's comparison theorem between \'etale and singular cohomology to $X_n$, and not to the quotient $[X_n/\SS_n]$. One reason is psychological --- the scheme $X_n$ is easier to visualize than the bug-eyed algebraic space $[X_n/\SS_n]$ --- but another reason is that it makes the argument significantly technically simpler. The reason is that a finite type algebraic space over $\C$ admits an analytification if and only if it is locally separated \cite[Theorem 2.2.5]{conradtemkin}, so the space $[X_n/\SS_n]$ doesn't even \emph{have} an analytification to which Artin's theorem could have been applied. The fact that $[X_n/\SS_n]$ does not admit an analytification is perhaps not surprising, if one accepts that no sensible map of analytic spaces could possibly correspond to the bug-eyed cover $[X_n/\SS_n] \to \A^n$. Thus if we wanted to apply Artin's theorem to $[X_n/\SS_n]$ we would have to form the analytification in a larger category of groupoids in complex analytic spaces, to which a version of Artin's theorem could have been applied only after a cohomological descent argument (which to our knowledge is not written out in the literature). 
\end{rem}
\subsection{General Weyl groups}
\begin{para}
An advantage of our geometric approach to Hyde's theorem is that it generalizes to a general Weyl group, as we now discuss. Suppose that we are given an $n$-dimensional root system, and that $W$ is the associated reflection group acting on $\A^n$. We will associate to this data a nonseparated scheme $X_W$ obtained by gluing together copies of $\A^n$ along complements of reflecting hyperplanes, in such a way that when $W$ is the symmetric group $\SS_n$ acting on $\A^n$ by permuting coordinates then we recover the scheme $X_n$. The correct definition of $X_W$ is clear once one realizes that a total order on the set $\{1,\ldots,n\}$ is precisely the same thing as a Weyl chamber for the symmetric group $\SS_n$.  More generally, a partition of the set $\{1,\ldots,n\}$ corresponds to a parabolic subgroup of $\SS_n$, and a total order on each block of the partition is a Weyl chamber for that parabolic subgroup. (We recall that a \emph{parabolic subgroup} of a reflection group $W$ is a subgroup which is the stabilizer of some point of $\A^n$.) Thus we define in general $X_W$ to be the nonseparated scheme parametrizing a point $y \in \A^n$ together with a Weyl chamber for the parabolic subgroup $\mathrm{Stab}(y)$. The stabilizer $\mathrm{Stab}(y)$ is itself a reflection group; it is the subgroup generated by reflections in hyperplanes containing $y$, so we can think of the data of a Weyl chamber for $\mathrm{Stab}(y)$ as a chamber of the sub-arrangement consisting of the reflecting hyperplanes through $x$. Just as $X_n$ is glued from $n!$ copies of $\A^n$, we may glue together $X_W$ from $\vert W \vert$ copies of $\A^n$, one for each Weyl chamber of $W$. 
\end{para}

\begin{para}
Let us now discuss how to generalize the statement of Hyde's theorem. Indeed, in formulating Hyde's theorem we used that to any element of $\Pol_n(\Fq) = (\A^n/\SS_n)(\Fq)$ we can associate a well defined conjugacy class in $\SS_n$. In the squarefree case we could take this conjugacy class to be the action of Frobenius on the set of roots of the polynomial, i.e.\ on the fiber of $\A^n \to \A^n/\SS_n$, but for a polynomial with repeated roots this does not prescribe a unique conjugacy class and our recipe was instead that we took the conjugacy class specified by the partition given by the degrees of the irreducible factors. How should this procedure be generalized?
\end{para}
\begin{para}\label{torsor}
First, we take a step back and consider a finite group $G$ acting \emph{freely} on a scheme $X$ over $\Fq$. Then any $\Fq$-point $x$ of $X/G$ determines a conjugacy class of elements of $G$. Indeed we lift $x \in (X/G)(\Fq)$ to a point $y \in X(\overline{\mathbf F}_q)$, and there is a unique element $g \in G$ such that $g \cdot y = \mathrm{Frob}(y)$. If we had chosen another lift $y$ then the element $g$ would have been changed by conjugation, so the conjugacy class of $g$ is independent of choices made. Alternatively, the conjugacy class is the image of Frobenius under the composite
$$ \widehat\Z \cong \pi_1^{\text{\emph{\'et}}}(\mathrm{Spec}(\Fq)) \stackrel x \longrightarrow \pi_1^{\text{\emph{\'et}}}(X/G) \to G,$$
where we omit basepoints and the second homomorphism classifies the $G$-torsor $X \to X/G$. In any case, if $G$ does not act freely on $X$ then we can still lift  $x \in (X/G)(\Fq)$ to $y \in X(\overline{\mathbf F}_q)$, but now an element $g \in G$ carrying $y$ to $\mathrm{Frob}(y)$ is only well defined as a left coset modulo the subgroup $\mathrm{Stab}(y)$; we do not get a well defined conjugacy class in $G$ but rather a conjugacy class of cosets, and in general this is the best we can say. 
\end{para}
\begin{para}\label{minimalparabolic}

Now suppose however that we are given a Weyl group $W$ acting on $\A^n$ as above. As already mentioned, a parabolic subgroup $\mathrm{Stab}(y)$ is itself a reflection group; it is the subgroup generated by reflections through hyperplanes containing $y$. In particular $\mathrm{Stab}(y)$ acts freely and transitively on the set of chambers of the arrangement of reflecting hyperplanes through $y$. An immediate consequence is that if we fix the choice of a Weyl chamber $C_0$ for $W$, then any coset of a parabolic subgroup has a distinguished representative: every left coset of $\mathrm{Stab}(y)$ contains a unique element $w$ with the property that $w \cdot C_0$ and $C_0$ are on the same side of every reflecting hyperplane passing through $y$. This element $w$ is called the \emph{minimal representative} of the coset. It can also be characterized as the unique element of the coset of smallest Coxeter length with respect to $C_0$. See e.g. \cite[I.5-1, I.5-2]{kane}. It follows that in the Weyl group case we \emph{can} in a canonical manner associate to an $\Fq$-point $x$ of $\A^n/W$ a conjugacy class of elements of $W$.  Indeed, we may lift $x$ to a point $y \in \A^n(\overline{\mathbf F}_q)$, there is a unique left coset modulo $\mathrm{Stab}(y)$ carrying $y$ to $\mathrm{Frob}(y)$, and this coset has a unique minimal representative $w$ once we choose a base chamber $C_0$. Choosing a different base chamber or a different lift would have the effect of changing $w$ by conjugation, so we get a well defined conjugacy class in $W$ associated to $x$. If $\chi \colon W \to \C$ is a class function then we denote $\chi(x)$ the value of $\chi$ on the conjugacy class associated to $x \in (\A^n/W)(\Fq)$. 
\end{para}
\begin{para}
We may now state the general form of Hyde's result. We state the result for a Weyl group for the convenience of having everything defined over $\Z$; for a general finite Coxeter group the scheme $X_W$ would be defined over the ring of integers of a number field and the statement of the following theorem would become more complicated.  
\end{para}

\begin{thm}Let $W \subset \mathrm{GL}(n, \Z)$ be a Weyl group acting by reflections on affine $n$-space. Let $\chi \colon W \to \C$ be a class function. Let $M_W$ be the complement of the associated codimension 3 subspace arrangement in $(\R^3)^n$. For any finite field $\Fq$ we have
\[ \sum_{x \in (\A^n/W)(\Fq)} \chi(x) = q^n \sum_{i\geq 0} q^{-i} \Big \langle \chi,\operatorname{ch} H^{2i}(M_W,\Q)\Big\rangle_{W}. \]
\end{thm}

\begin{para}
 The proof follows just as the proof of Hyde's theorem, by applying the Grothendieck--Lefschetz trace formula to $[X_W/W]$ and combining the following ingredients:
\begin{itemize}
\item The $W$-equivariant weak homotopy equivalence $M_W \simeq X_W(\C)$ (\cref{weakequivalence}).
\item The space $X_W$ admits an algebraic cell decomposition over the integers, so that its cohomology is pure of Tate type, with Betti numbers independent of characteristic of the base field (\cref{exists cell decomposition} and \cref{cellprop}).
\item $W$ acts freely on $X_W$, and the algebraic space $[X_W/W]$ is a bug-eyed cover of $\A^n/W\cong \A^n$, so that there is a bijection of $\Fq$-points $[X_W/W](\Fq) \stackrel\sim\longrightarrow (\A^n/W)(\Fq)$.
\item Under this bijection, the $W$-conjugacy class associated to an element of $[X_W/W](\Fq)$ (as in \S\ref{torsor}, using that $W$ acts freely on $X_W$) concides with the conjugacy class associated to the corresponding element of $(\A^n/W)(\Fq)$ (as in \S\ref{minimalparabolic}).
\end{itemize}
The first three points work out in exactly the same way as for Hyde's theorem, so let us only comment on the last point. Pick a point $x \in (\A^n/W)(\Fq) = [X_W/W](\Fq)$. An $\overline{\mathbf F}_q$-point of $X_W$ over $x$ is given by a pair $(y,C)$ where $y \in \A^n(\overline{\mathbf F}_q)$ is in the fiber of $\A^n \to \A^n/W$ over $x$, and $C$ is a Weyl chamber of $\mathrm{Stab}(y)$ acting on $\A^n$. The Galois group acts as 
$$ \mathrm{Frob}((y,C)) = (\mathrm{Frob}(y),C)$$
and since $W$ acts freely on $X_W$ there is a unique $w \in W$ such that $w \cdot (y,C) = (\mathrm{Frob}(y),C)$. The claim is that if we choose a base chamber $C_0$ for the Coxeter group $W$ contained inside $C$ then $w$ is the minimal representative inside the coset consisting of elements which move $y$ to $\mathrm{Frob}(y)$. Indeed, $w \cdot C_0$ and $C_0$ being on the same side of every reflecting hyperplane for the parabolic subgroup $\mathrm{Stab}(y)$ implies that $w$ preserves the Weyl chamber of this parabolic subgroup containing $C_0$, i.e. \ $C$. This completes the proof of the generalized version of Hyde's theorem. 
\end{para}

\begin{rem}A cute application of our construction for a general reflection group is as follows. Consider a finite reflection group $W$ acting on affine space. Consider on one hand our space $X_W$, which parametrizes pairs $(y,C)$ of a point in affine space and $C$ a chamber of $\mathrm{Stab}(y)$, and on the other hand the ``inertia variety'' $I_W$ parametrizing pairs $(y,w)$ with $y$ a point in affine space and $w$ an element of $\mathrm{Stab}(y)$. Since $\mathrm{Stab}(y)$ acts freely and transitively on its set of chambers we obtain a set-theoretic bijection between $X_W$ and $I_W$ after choosing a base chamber $C_0$ of $W$: indeed, such a choice fixes also the choice of a base chamber for each parabolic $\mathrm{Stab}(y)$, and hence a bijection between $\mathrm{Stab}(y)$ and its set of chambers for every point $y$. Now
$$ I_W \cong \coprod_{w \in W} \mathrm{Fix}(w),$$
where $\mathrm{Fix}(w) \subseteq \mathbb A^n$ denotes the linear subspace fixed by $w$, and each subset $\mathrm{Fix}(w)$ is carried isomorphically to a locally closed subscheme of  $X_W$ under this bijection. In particular we may partition both $X_W$ and $I_W$ into locally closed subschemes which are pairwise isomorphic; both spaces $X_W$ and $I_W$ have the same class in the Grothendieck ring of varieties, and we obtain a calculation of $[X_W] \in K_0(\mathrm{Var})$. Thus for any algebraic cell decomposition of $X_W$ the number of $d$-dimensional cells in the decomposition is just the number of $w \in W$ whose set of fixed points has dimension $d$. Since the Betti numbers of the complement of the arrangement of reflecting hyperplanes can be read off from a cell decomposition we obtain the formula
$$ \sum_{w \in W} q^{\mathrm{codim}\,\mathrm{Fix}(w)} = \sum_i \dim H^{2i}(M_W,\Q) \cdot q^i.$$
The result is of course not new: both sides are equal to $\prod_i (1+m_iq)$, where $m_i$ are the exponents of the group, by classical results of Shephard-Todd \cite[Theorem 5.3]{shephardtodd} and Brieskorn \cite[Th\'eor\`eme 6(ii)]{brieskorn}, respectively. Strictly speaking Brieskorn's results concern the cohomology of the complement of the \emph{complexified} arrangement, so we also need to know that the complements of the arrangements in $\C^n$ and $(\R^3)^n$ have the same Betti numbers up to a degree shift, which is for example a consequence of the Goresky--MacPherson formula \cite{goreskymacpherson}. \end{rem}

\begin{rem}The construction of the previous paragraph produces a decomposition of $X_W$ into pairwise disjoint locally closed subschemes isomorphic to affine spaces, indexed by elements $w \in W$, after the choice of a base chamber. At an earlier stage of this project we expected this decomposition to be an algebraic cell decomposition in the sense of \cref{cell def}. This is not true, however: for the scheme $X_4$, the ``cells'' corresponding to the permutations $(14)$ and $(13)(24)$ both have the property that they intersect each other's closure nontrivially. According to \cref{exists cell decomposition} it is still true that $X_W$ admits an algebraic cell decomposition, but the construction is more nontrivial. 
\end{rem}

\begin{rem}As explained by Hyde \cite[Section 2.5]{hyde}, one can evaluate the $q = 1$ specialization of Hyde's theorem to produce an expression for the direct sum of all cohomology groups of $\PConf_n(\R^3)$ as a representation of $\SS_n$, and the result is that 
$$ \bigoplus_k H^{k}(\PConf_n(\R^3),\Q) \cong \Q[\SS_n],$$
i.e.\ one obtains the regular representation. (As Hyde points out, the result is not new.) It is natural to ask whether our geometric approach can shed light on this observation, in particular given that setting $q=1$ amounts to counting the number of cells in an algebraic cell decomposition. Unfortunately we do not see an easy way to see this result (in particular since the cell decomposition is not equivariant), but we remark that the analogous statement remains true for any finite reflection group; see for example \cite[Theorem 1.4(b)]{moseley}. Let us give a short proof. Suppose $W$ acts by reflections on $\R^n$, and for any $d \geq 1$ let $M^{(d)}$ denote the associated codimension $d$ subspace arrangement in $(\R^d)^n$. We claim that the $W$-equivariant Euler characteristic of $M^{(d)}$ depends only on the parity of $d$. Indeed, writing $\R^d = \C \times \R^{d-2}$ we get an induced $\C^\ast$-action on $M^{(d)}$ with fixed point set $M^{(d-2)}$, and by localization the two Euler characteristics coincide. More explicitly, one may apply Mayer--Vietoris to  the open cover of $M^{(d)}$ given by the complement of $M^{(d-2)}$ (whose Euler characteristic vanishes because $\C^\ast$ acts freely) and the $W \times \C^\ast$-equivariant tubular neighborhood of $M^{(d-2)}$ given by $\C^n \times M^{(d-2)}$. When $d=1$ we have that $M^{(1)}$ is a disjoint union of contractible components which are freely and transitively permuted by $W$, so in this case the equivariant Euler characteristic is the regular representation of $W$. Then the same is true for $M^{(3)} = M_W$, and since $M_W$ has no odd cohomology the claim follows. 
\end{rem}

\subsection{Type B and even polynomials}

\begin{para}
Let us examine our formula more closely in the type $B$ case. The type $B$ Weyl group is the hyper\-octahedral group $B_n \coloneqq \SS_n \ltimes \Z/2^n$,  which we can think of as the automorphism group of the $(\Z/2)$-set $[n]^\pm = \{1,-1, \dots, n,-n\}$.  The group $B_n$ acts on $\A^n = \Spec \Z[z_1, \dots, z_n]$,   and its ring of invariants is well known to be generated by
$$w_d(z) \coloneqq \sum_{1 \leq i_1 < \dots < i_d \leq n}  z_{i_1} ^2 \dots z_{i_d}^2$$
for $d = 1, \dots, n$. These invariants give coordinates on the quotient space $\A^n/B_n = \Spec(\Z[w_1, \dots, w_n])$.  We can interpret the coordinate  $w_d$ as the coefficient of  $t^{2d}$ in the product
$$\prod_{i = 1}^n  (t^2 - z_i^2)  = \prod_{i = 1}^n (t-z_i)(t+z_i).$$
Over an algebraically closed field, a polynomial $f(t)$ admits a factorization of this form if and only if it is \emph{even}.   Thus we see that ${\A}^n/B_n$ is a parameter space for monic degree $2n$  even polynomials, where $w_d$ records the $2d$th coefficient, and ${\A}^n \to {\A}^n/B_n$ is the finite map that takes $(z_1, \dots, z_n)$ to the even monic degree $2n$ polynomial $\prod_i  (t^2 - z_i^2)$.  In these terms, our space $X_{B_n}$ parameterizes even monic degree $2n$ polynomials,  together with a choice of labelling of the set of roots by the set $[n]^\pm$, subject to the condition that $\alpha_{-j} = - \alpha_j$ for all $j \in [n]^\pm$.  
\end{para}

\begin{para}
Let $f \in \Pol_{2n}^{\mathrm{even}} (\Fq)$ be an even monic degree $2n$ polynomial.  Its factorization into irreducibles takes the form  $$f(t) = \prod_{g \text{ even}} g(t) \prod_{h \text{ not even}} h(t)h(-t),$$ where we have used that the factorization into irreducibles must be invariant under the transformation $t \mapsto -t$.  We associate to $f$ a double integer partition of $n$, which we call its \emph{factorization type}:  \[(\lambda, \mu) = (1^{u_1} 2^{u_2} \dots, \overline 1^{v_1}   \overline 2^{v_2} \dots),~  d = \sum_{i} u_i i + \sum_{j} v_j j. \]  Here $u_i$ is the number of degree $2i$ irreducible even factors of $f$, and $v_j$ is the number of pairs $\{h(t),h(-t)\}$ of degree $j$ irreducible non-even factors of $f$.  Similarly, a signed permutation $\sigma \in B_n$ gives rise to a $(\Z/2)$-equivariant partition of the set $[n]^\pm$ into cycles, and the cycle type of $\sigma$ may be encoded by a double partition, \[(1^{a_1} 2^{a_2} \dots, \overline 1^{b_1}  \overline 2^{b_1} \dots),\] 
where $a_i$ denotes the number of $(\Z/2)$-invariant blocks of size $2i$ in the partition, and $b_j$ denotes the number of pairs of blocks of size $j$ that are interchanged by the $(\Z/2)$-action. This construction determines a bijection between conjugacy classes of $B_n$ and double partitions.  A polynomial $f$ has factorization type $(\lambda, \mu)$  if and only if Frobenius permutes a signed labelling of its roots in $X_{B_n}(\overline{\mathbf F}_q)$ by an element $\sigma \in B_n$  in the conjugacy class $(\lambda, \mu)$.  
Thus we have proved the following.    
\end{para}

\begin{thm} 
Let $\chi$ be a character of $B_n$,  considered as a function on the set of double integer partitions of $n$.   We write $\chi(f)$ for the the value of $\chi$ on the factorization type of $f$. Then

\[ \sum_{f \in \Pol_{2d}^{\mathrm{even}} (\Fq)} \chi(f) = q^n \sum_{i\geq 0} q^{-i} \Big\langle \chi,\operatorname{ch} H^{2i}(\PConf_n^{\Z/2} (\R^3),\Q) \Big\rangle_{B_n}. \]

Here $\PConf_n^{\Z/2}(\R^3)$ denotes the orbit configuration space, which parameterizes $(\Z/2)$-equivariant injections of $[n]^{\pm}$ into $\R^3$, considered with its $(\Z/2)$-action $x \mapsto -x$.  
\end{thm}

\begin{para}Thus we see that in the type $B$ case,  our generalized version of Hyde's theorem  gives a formula for the factorization statistics of even polynomials. These statistics record the number of even and non-even factors in the factorization type of $p$.   As mentioned previously,  Jim\'enez Rolland and Wilson \cite{jimenezrollandwilson} studied a Type $B$ analog of squarefree factorization statistics.   Their formulation of factorization statistics is different, but equivalent to the one used here.   Jim\'enez Rolland--Wilson consider monic squarefree polynomials $q(t)$ of degree $d$ with nonzero constant term, and their statistics concern the factorization of $q(t)$ into irreducibles and whether the roots of $q$ are quadratic residues.   The statistics we consider are related to theirs by the transformation $f(t) = q(t^2)$.  Indeed, let $i(t)$ be a degree $d$ irreducible polynomial. Then $i(t^2)$ either remains irreducible or factors as $h(t)h(-t)$.  Further, $i(t^2)$ factors if and only if all of the roots of $i(t)$ in $\mathbf F_{q^{\deg i}}$  are quadratic residues, and $i(t^2)$ remains irreducible if and only if all of the roots of $i(t)$ are non-residues.  
\end{para}

\section{Non-Hausdorff models for complements of arrangements}\label{section three}

\subsection{Conventions and terminology on hyperplane arrangements}

\begin{para}Let $A$ be a finite set of affine hyperplanes in $\R^n$. We let $\L(A)$ denote the partially ordered set consisting of all nonempty intersections of hyperplanes in $A$, including $\R^n$ itself as the empty intersection, ordered by reverse inclusion. Elements of $\L(A)$ are called \emph{flats}. If $K \in \L(A)$ is a flat, then $K \setminus \bigcup_{K<L} L$ is a finite union of convex open regions inside $K$ called \emph{faces}. We denote by $\F(A)$ the set of all faces of the arrangement $A$. We partially order the set of faces according to 
$$ F \leq G \iff \overline F \supseteq G.$$
If $E$ is a subset of affine $n$-space (typically a flat), then we denote by $\Supp(E) \subseteq A$ the set of hyperplanes containing $E$. 
\end{para}
\begin{para}
The minimal elements of $\F(A)$ are called \emph{chambers}, and we denote the set of chambers by $\Ch(A)$. For $C, D \in \Ch(A)$ we denote by $\Sep(C,D) \subseteq A$ the subset of hyperplanes $H$ which separate $C$ from $D$, i.e.\ such that $C$ and $D$ are contained in different connected components of $V \setminus H$. 
\end{para}

\begin{para} We define $\mathrm{Strat}(A)$ to be the partially ordered set consisting of pairs $(K,C)$ where $K \in \L(A)$ and $C \in \Ch(\Supp(K))$. The partial order relation is given by $(K,C) \preceq (L,D)$ if $K \leq L$ and $C \supseteq D$. This poset will parametrize the strata for a natural stratification of the scheme $X_A$  which we will define shortly. \label{strat}
\end{para}

\begin{para}
The poset $\mathrm{Strat}(A)$ is an analogue of the \emph{Salvetti poset} $\Sal(A)$, which is the partially ordered set consisting of pairs $(F,C)$ where $F \in \F(A)$ and $C \in \Ch(\Supp(F))$. The partial order relation is given by $(F,C) \preceq (G,D)$ if $F \leq G$ and $C \supseteq D$. Usually in the literature elements of $\Sal(A)$ are described as pairs $(F,C)$ where $F \in \F(A)$ and $C \in \Ch(A)$ is a chamber adjacent to $F$ (i.e.\ whose closure contains $F$); this is equivalent to the definition given here, since every chamber of $\Supp(F)$ contains a unique chamber of $A$ that is adjacent to $F$. This poset was introduced by Salvetti \cite{salvetti}, who proved that the complement of the complexification of the arrangement $A$ deformation retracts onto a regular CW complex whose poset of cells is precisely the Salvetti poset. The similarity between the Salvetti poset and our stratification of $X_A$ is no coincidence, as we explain in \cref{connection to salvetti}.  \label{salvetti poset}
\end{para}

\subsection{Construction of the scheme \texorpdfstring{$X_A$}{XA}}

\begin{para}
Let $A$ be as above. Following Proudfoot \cite{proudfoot}, we are going to associate to $A$ a nonseparated scheme $X_A$ over $\Spec(\R)$, which will be obtained by gluing together a copy of the affine space $\A^n$ for each chamber of the arrangement. If all hyperplanes in $A$ are defined over a subring $R \subset \R$ then $X_A$ is defined over $R$, too. 
\end{para}
\begin{para}
Let us first recall Zariski gluing of schemes in general. Let $X = \bigcup_{i \in I} X_i$ be an open cover of a scheme $X$. Let $X_{ij} = X_i \cap X_j$. Then the identity map on $X$ induces a collection of isomorphisms $X_{ij} \stackrel \sim \to X_{ji}$ for all $i$ and $j$, and $X$ can be uniquely reconstructed from the collection $\{X_i\}$ and the isomorphisms $X_{ij} \cong X_{ji}$. In the other direction, let $\{X_i\}_{i \in I}$ be an arbitrary family of schemes. Suppose we are given
 Zariski open subset $X_{ij} \subset X_i$ for every $i, j \in I$, and an isomorphism $\phi_{ij} \colon X_{ij} \to X_{ji}$ for every $i, j \in I$.
This family of schemes $\{X_i\}$ may be glued together along the isomorphisms $\phi_{ij}$ to a scheme $X$, which is then unique up to canonical isomorphism, if and only if the gluing data satisfies the following two axioms:
\begin{itemize}
\item $X_{ii} = X_i$, and $\phi_{ii} = \mathrm{id}$ for all $i \in I$,
\item $\phi_{ij}$ restricts to an isomorphism from $X_{ij} \cap X_{ik}$ to $X_{ji} \cap X_{jk}$, and $\phi_{jk} \circ \phi_{ij} = \phi_{ik}$ on $X_{ij} \cap X_{ik}$, for every $i, j, k \in I$ (cocycle condition). 
\end{itemize} 
\end{para}
\begin{para}\label{construction}
We can now write down gluing data for the scheme $X_A$. For every $C \in \Ch(A)$ we let $X_C = \A^n$. For any pair $D, C \in \Ch(A)$ we define $$X_{CD} = \A^n \setminus \left(\bigcup_{H \in \Sep(C,D)} H\right) \subseteq X_C.$$ The identity map on $\A^n$ gives an isomorphism $\phi_{CD} \colon X_{CD} \cong X_{DC}$, and the resulting gluing data  is easily checked to satisfy the cocycle condition since $$\Sep(C,E) \subset \Sep(C,D) \cup \Sep(D,E).$$ We denote the resulting scheme by $X_A$. 
\end{para}

\begin{remark}
A reader may wonder why we took care to say that we glued together copies of the scheme $\A^n$ to obtain a scheme over $\Spec(\R)$, instead of just saying that we are gluing together copies of $\R^n$ or $\C^n$. The reason is that once we have a \emph{scheme} $X_A$ over a base ring $R$ then we can talk freely about its set of points over an arbitrary $R$-algebra. In particular, the scheme $X_A$ will be defined over $\Z$ whenever all the hyperplanes have integer coefficients, and in this case we can make sense of $X_A(\Fq)$ for any prime power $q$, which is at the core of our intended application to factorization statistics. If $K$ is a field and an $R$-algebra, then a $K$-valued point of $X_A$ is an element $(b_1,\ldots,b_n)$ of $K^n = \A^n(K)$, together with the choice of a chamber of $\Supp(\{(b_1,\ldots,b_n)\})$, i.e.\ the sub-arrangement of $A$ consisting of all hyperplanes containing $(b_1,\ldots,b_n)$. \end{remark}

\subsection{Construction of the fiber bundle}

\begin{para}
Let $A$ be a finite arrangement of hyperplanes in $V \cong \R^n$. We write $V_\C$ for the complexification of $V$, and $A_\C = \{ H_\C : H \in A\}$ for the induced arrangement of complex hyperplanes in $V_\C$. The complement of $A_\C$ in $V_\C$ will be denoted $M(A_\C)$. Similarly we write $V_{(3)} = V_\C \times V$, we set $A_{(3)}$ for the codimension $3$ subspace arrangement consisting of the subspaces $\{ H_\C \times H : H \in A\}$, and we let $M(A_{(3)})$ denote the complement of this codimension $3$ arrangement. 
\end{para}

\begin{ex}If $A$ consists of the diagonal hyperplanes $x_i=x_j$ in $\R^n$, then $M(A_\C)$ is the configuration space of $n$ distinct ordered points in $\C$ and $M(A_{(3)})$ is the configuration space of $n$ distinct ordered points in $\R^3$. \end{ex}

\begin{para}
If $C \in \Ch(A)$ is a chamber, then  define 
\begin{equation*}U_C = \left\{ (x,z) \in  V_\C \times V : \hspace{.5em}\parbox{15em}{if $x \in H_\C$, then $z$ is in the \\ component of $V \setminus H$ containing $C$}\right\}.\end{equation*}
Observe that the collection $\{U_C\}_{C \in \Ch(A)}$ is an open cover of $M(A_{(3)})$. Write $f$ for the projection from $V_\C \times V$ onto $V_\C$. Note that 
\begin{equation} \label{gluing condition for map} f(U_C \cap U_D) = V_\C \setminus\bigcup_{H \in \Sep(C,D)} H_\C.\end{equation}
\end{para}\begin{para}\label{construction of map}
Recall from \S\ref{construction} the open cover $X_A = \bigcup_{C \in \Ch(A)} X_C$ of the scheme $X_A$, where each $X_C$ is an affine space $\A^n$. We will now construct a continuous map $M(A_{(3)}) \to X_A(\C)$: 
\begin{itemize}
\item Cover $M(A_{(3)})$ by the open sets $U_C$. 
\item Map each open set $U_C$ onto $V_\C \cong X_C(\C)$ via $f$. 
\item By comparing Equation \eqref{gluing condition for map} and the description of the intersections $X_C \cap X_D = X_{CD}$ from  \S\ref{construction}, observe that these glue to a map $M(A_{(3)}) \to X_A(\C)$.
\end{itemize}\end{para}

\begin{thm}The map $M(A_{(3)}) \to X_A(\C)$ is a fiber bundle with fiber $\R^n$. \label{weakequivalence}
\end{thm}

\begin{proof}The question is local on the base, so it suffices to show that we get a fiber bundle over each set of the open cover $X_A(\C) = \bigcup_{C \in \Ch(A)} X_C(\C)$, i.e.\ that $U_C \to V_C  \cong \C^{n}$ is a fiber bundle for every $C \in \Ch(A)$. Now $U_C$ is an open subset of the trivial vector bundle $\R^{3n} \cong V_\C \times V  \to V_\C$. Moreover, the fiber of $U_C \to V_\C$ over an element $x \in V_\C$ is nothing but the chamber of the arrangement $\Supp(x)$ containing $C$, considered as an open subspace of $V$. In particular the fibers are nonempty and convex, which by the following \cref{bundle lemma} implies that the map is a fiber bundle. \end{proof}

\begin{lem}\label{bundle lemma}Let $p \colon E \to X$ be a rank $n$ vector bundle over a metrizable topological space, and $U \subset E$ an open subspace such that $U \cap p^{-1}(x)$ is convex and nonempty for all $x\in X$. Then $U \to X$ is an $\R^n$-fiber bundle, and $U$ is homeomorphic to $E$ as a bundle over $X$. 
\end{lem} 

\begin{proof}By  a partition of unity there is a section contained in $U$, which we may assume to be the zero section. The strategy will be to construct an exhaustion of $U$ by fiberwise compact subsets
$$ \{0\} \Subset B_1 \Subset B_2 \Subset B_3 \Subset \ldots \qquad \bigcup_{k \geq 1} B_k = U,$$
together with a compatible system of homeomorphisms between $B_k$ and the disk bundle of radius $k$ inside $E$, for some bundle metric which we fix from now on. Let $S \to X$ be the unit sphere bundle inside $E$. Define a map $\phi \colon S \to \R_{>0} \cup \{\infty\}$ by
$$ \phi(z) = \sup \{ t \in \R : t \cdot z \in U\},$$
so  that if we identify $E^\ast$ (the complement of the zero section in $E$) with $S \times \R_{>0}$ then 
$$ U^\ast = \{(z,t) \in S \times \R_{>0} : t  <\phi(z)\}.$$
Since $U$ is open, $\phi$ is lower semicontinuous, and then there exists a strictly increasing sequence of continuous functions $\psi_1,\psi_2,\ldots \colon S \to \R_{>0}$ converging pointwise to $\phi$ (see e.g.\ \cite[Theorem 2.1.3]{ransford}). Define
$$ B_k^\ast =  \{(z,t) \in S \times \R_{>0} : t  \leq \psi_k(z)\}.$$
If we then let $B_k$ be the union of $B_k^\ast$ with the zero section then each $B_k$ is homeomorphic to the radius $k$ disk bundle by rescaling, these homeomorphisms may be chosen compatibly with each other, and by construction the bundles $B_k$ exhaust $U$. \end{proof}
 
\begin{rem}It is not hard to see that the construction of \cref{bundle lemma} can with only a little extra care be done smoothly, so that if the base  $X$ were a smooth manifold then $E$ and $U$ would be diffeomorphic. \end{rem}

\begin{rem}\label{remark on real points}The arguments given here generalize to show that similarly $M(A_\C)$ is an $\R^n$-fiber bundle over the set of \emph{real} points $X_A(\R)$ of $X_A$, which recovers (by a somewhat different argument) a theorem of Proudfoot \cite{proudfoot}. Even more generally, the induced $c$-arrangement in $(\R^c)^n$ is an $\R^n$-bundle over a non-Hausdorff topological space obtained by gluing together copies of the euclidean space $(\R^{c-1})^n$ along complements of subspaces in exactly the same way that we glued together $X_A$ from copies of affine space $\A^n$. The only difference when $c>3$ is that the base of the fiber bundle can not be interpreted as the real or complex points of an algebraic scheme.  \end{rem}

\subsection{Stratification of \texorpdfstring{$X_A$}{XA}} 

\begin{para}
Let $X$ be a topological space (or a scheme). A \emph{stratification} of $X$ is a decomposition of $X$ into finitely many nonempty pairwise disjoint locally closed subspaces (subschemes) $X_\alpha$, which are called \emph{strata}, such that the closure of a stratum is a union of strata. We partially order the set of strata by the rule
$$ \alpha \preceq \beta \iff \overline X_\alpha \supseteq X_\beta. $$
Observe that a stratification is uniquely determined by the set of \emph{closed strata} $\overline X_\alpha$, and we will sometimes find it convenient to describe a stratification of a space by specifying the closed strata. 
\end{para}\begin{para}\label{strat2}
In particular, if $A$ is an arrangement of hyperplanes in $\R^n$, then the collection of flats $\L(A)$ constitute the set of closed strata for a stratification of $\A^n$ into subschemes. When we form the scheme $X_A$ then we glue together copies of $\A^n$ along unions of strata, which means in particular that the stratification of $\A^n$ induces a stratification of the scheme $X_A$. Explicitly, we can understand $X_A$ and its stratification by starting from the scheme 
$$ \coprod_{C \in \Ch(A)} \A^n,$$
together with its natural stratification with strata indexed by $\L(A) \times \Ch(A)$. To obtain the scheme $X_A$ we need to impose the gluing relation on strata that the stratum $(K,C)$ is identified with $(L,D)$ if $K=L$, and the chambers $C$ and $D$ are not separated by any hyperplane containing $K$. Thus the poset of strata of $X_A$ is the poset $\mathrm{Strat}(A)$ introduced in \S\ref{strat}. 
 \end{para}

\begin{rem}\label{connection to salvetti}Each stratum $(K,C)$ of $X_A$ is itself the complement of an arrangement of hyperplanes inside a vector space of dimension $\dim (K)$. This implies in particular that the set of \emph{real} points of each stratum in $X_A$ is disconnected, with each connected component contractible. More precisely a connected component of the real points of the stratum $(K,C)$ is just a face in $\F(A)$ which is open in $K$. Thus $X_A(\R)$ has a more refined stratification indexed by pairs $(F,C)$ with $F \in \F(A)$ and $C \in \Ch(\Supp(F))$, so we have recovered  the Salvetti poset of \S\ref{salvetti poset}. Now it is not hard to verify that the closures of strata are contractible and that their intersections are empty or contractible, so that by the Nerve Lemma it follows that $X_A(\R)$ is weakly homotopy equivalent to the order complex of $\Sal(A)$. On the other hand we have the $\R^n$-fiber bundle $M(A_\C) \to X_A(\R)$ of \cref{remark on real points}, so $M(A_\C)$ is homotopic to the order complex of $\Sal(A)$, as originally proven by Salvetti. A more direct way of understanding this connection is that the stratification of $X_A(\R)$ pulls back to a stratification of $M(A_\C)$, and the stratum inside $M(A_\C)$ corresponding to a pair $(F,C) \in \Sal(A)$ is precisely the cartesian product 
$$F \times C.$$
This coincides with the stratification of $M(A_\C)$ described by Bj\"orner--Ziegler in terms of ``sign vectors'' \cite{bjornerziegler}.  \end{rem}

\subsection{Algebraic cell decomposition}

\begin{defn}\label{cell def}An \emph{algebraic cell decomposition} of a scheme $X$ is a filtration of $X$ by open subschemes
$$ \varnothing = U^{(0)} \subset U^{(1)} \subset \ldots \subset U^{(t)} = X$$
where for each $j$ there is an isomorphism $U^{(i)} \setminus U^{(i-1)} \cong \A^{d_i}$. We call $U^{(i)} \setminus U^{(i-1)}$ a \emph{cell} of dimension $d_i$. 
\end{defn}
\begin{para}\label{observations}
We will now show that $X_A$ admits an algebraic cell decomposition. Our strategy will be to make use of the open cover $X_A = \bigcup_{C \in \Ch(A)} X_C$ considered in \S\ref{construction}, where $X_C \cong \A^d$ for all $C \in \Ch(A)$. To be more precise, suppose that $C_1,C_2,\ldots,C_t$ is an enumeration of the elements of $\Ch(A)$ in some order. Define for each $0 \leq i\leq t$ the open subset $U^{(i)} \subseteq X_A$ by the formula
$$ U^{(i)} \coloneqq X_{C_1} \cup X_{C_2} \cup \ldots \cup X_{C_i}.$$
Clearly we have a chain of inclusions
$$ \varnothing = U^{(0)} \subset U^{(1)} \subset \ldots \subset U^{(t)} = X_A, $$
so all that is needed for this to be an algebraic cell decomposition is that $U^{(i)} \setminus U^{(i-1)}$ is isomorphic to an affine space for every index $i$. Now we have
$$ U^{(i)} \setminus U^{(i-1)} = X_{C_i} \setminus \bigcup_{j<i} (X_{C_i} \cap X_{C_j}),$$ and if we also note that
\begin{equation} \label{description of intersections} X_{C_i} \cap X_{C_j} = \A^d \setminus \bigcup_{H \in \Sep(C_i,C_j)} H, \end{equation}
then we see that an enumeration $C_1,C_2,\ldots,C_t$ of the chambers of the arrangement $A$ will give us an algebraic cell decomposition if and only if the following condition is satisfied:
\[ \text{ ($\star$) for any $i=1,\ldots,t$, we have that } \,\,\bigcap_{j<i} \,\, \bigcup_{H \in \Sep(C_i,C_j)} H\,\, \text{ is an affine subspace of $\A^d$.}\]
\end{para}
\begin{defn}\label{valid order}An enumeration $C_1,C_2,\ldots,C_t$ of $\Ch(A)$ is called a \emph{valid order} if the condition ($\star$) is satisfied. \end{defn}
\begin{para}
Perhaps surprisingly, the notion of a valid order is not new; when expressed combinatorially,  it coincides with a notion introduced by Delucchi. One has the following theorem:
\end{para}

\begin{thm}[Delucchi, Lofano--Paolini]Every real hyperplane arrangement admits a valid order. \label{delucchi}\end{thm}

\begin{para} \cref{delucchi} was proven by Delucchi in the case of a central hyperplane arrangement \cite[Theorem 4.15]{delucchi}. The general case of \cref{delucchi} is due to Lofano--Paolini \cite[Theorem 5.6]{lofanopaolini}, who also introduced the terminology ``valid order''. Lofano--Paolini's construction of valid orders is particularly simple: let $x_0$ be a generic point of our ambient real vector space, and suppose that the chambers are enumerated in any weakly increasing order by their euclidean distance from the point $x_0$, i.e.\
$$ d(C_i,x_0) < d(C_j,x_0) \implies i<j.$$
Then this is always a valid order. We remark that there is a small gap in Delucchi's argument for \cref{delucchi}: he uses an induction on the number of hyperplanes in the arrangement, and that a total order on $\Ch(A)$ induces a total order on $\Ch(A \setminus \{H\})$, for any $H \in A$. Unfortunately the family of total orders considered by Delucchi is not closed under the operation of removing a hyperplane, which breaks the inductive strategy. Delucchi has communicated to us that if one runs the same argument using e.g.\ the euclidean orderings of Lofano--Paolini (which are closed under removing a hyperplane) then the argument works. 
\end{para}

\begin{para}
Combining \cref{delucchi} with the observations of \S\ref{observations}, we immediately deduce the following:
\end{para}
\begin{cor}For any real hyperplane arrangement $A$, the associated nonseparated scheme $X_A$ admits an algebraic cell decomposition. \label{exists cell decomposition}
\end{cor}

\begin{rem}\label{refinement}
The cell decomposition of $X_A$ obtained from \cref{exists cell decomposition} is a \emph{refinement} of the  stratification of $X_A$, in the sense that each affine cell is a union of strata. In this way we can think of the algebraic cell decomposition as an increasing filtration of the poset $\mathrm{Strat}(A)$, such that the union of strata that are in the $k$th step of the filtration but not the $(k-1)$st form a cell. 
\end{rem}

\begin{rem}At first sight it is perhaps surprising that the question whether there always exists a total order on $\Ch(A)$ satisfying the strange condition ($\star$) has already been considered (and answered) in the literature. But in fact Delucchi's motivation for introducing the condition ($\star$) was not unrelated to ours. Recall that a CW complex is said to be \emph{minimal} if its $k$th Betti number coincides with the number of $k$-cells. By a theorem due independently to Randell \cite{randell} and Dimca--Papadima \cite{dimcapapadima}, the complement of any complex hyperplane arrangement has the homotopy type of a minimal CW complex, see also \cite{adiprasito}. On the other hand as we have already mentioned the complement of a complexified \emph{real} arrangement of hyperplanes deformation retracts onto an the \emph{Salvetti complex}, and it is natural to try to make minimality explicit in this case by writing down an explicit perfect discrete Morse function on the Salvetti complex. Delucchi \cite{delucchi} showed that a valid order on $\Ch(A)$ induces a perfect discrete Morse function on $\Sal(A)$, and this was his motivation for introducing the notion of a valid order. What we have explained above is precisely that a valid order on $\Ch(A)$ induces a filtration of the poset $\mathrm{Strat}(A)$ --- and hence also of $\Sal(A)$ --- such that the union of strata within each successive difference is homeomorphic to an open cell. Delucchi showed instead that the poset of strata within each successive difference admits an acyclic matching with a single critical cell, so that these can be ``glued together'' to a discrete Morse function on the whole poset $\Sal(A)$. For another approach to minimality via discrete Morse theory on the Salvetti complex see  \cite{salvettisettepanella}. 

A related fact is that the filtration that we obtain is ``minimal'' in the sense that the cells form a basis for the Borel--Moore homology of the complement of the arrangement. See \cite{itoyoshinaga} for another quite similar construction of such stratifications of complements of complexified real hyperplane arrangements. In fact it is not hard to show that more is true: the cells make up a CW decomposition of the one-point compactification of the complement, in which the attaching maps are nullhomotopic by a straight-line homotopy, so that the one-point compactification is a wedge of spheres in a very explicit way. This works for both $M(A_{(3)})$ and $M(A_\C)$. 
  \end{rem}

\begin{rem}
One can think of the notion of a valid ordering as an analogue of the notion of a shelling of a simplicial complex. Let us explain this. We consider all posets to have the Alexandroff topology, in which the open subsets are the downwards closed subsets. Then the poset $\Strat(A)$ has an evident open cover indexed by $\Ch(A)$, with each open subset isomorphic to $\L(A)$: explicitly, to $C \in \Ch(A)$ we associate the subposet
$$ S_C = \{ (K,D) \in \Strat(A) : C \subseteq D\} \cong \L(A). $$
The open cover $\Strat(A) = \bigcup_{C \in \Ch(A)} S_C$ corresponds geometrically to the open cover of the scheme $X_A$ by the subschemes $X_C$ of \S\ref{construction}. In these terms we can define a valid order as an enumeration of these open sets, say $S_1,\ldots, S_t$, such that the poset 
$$ S_i \setminus \bigcup_{j<i} S_j$$
has a unique minimum element for all $i$. If instead $F$ is the face poset of a finite purely $d$-dimensional simplicial complex, then $F$ has an open cover by subposets each isomorphic to the boolean poset of rank $d$ (corresponding to the maximal dimensional faces), and a shelling is an enumeration of the sets in this open cover satisfying precisely the same condition. However, we will not pursue this analogy further here. 
\end{rem}

\subsection{Cohomology of non-separated schemes via algebraic cell decompositions}

\begin{para}In this section we explain why the cohomology of a smooth scheme with an algebraic cell decomposition (\cref{cell def}) is torsion free, and why one can read off its Betti numbers from the number of cells in each dimension. This result is certainly well known, but we do not know a reference stating it in the generality we need it, i.e.\ for the \'etale cohomology of potentially nonseparated schemes. But let us first recall the (easy) argument in the case of a complex algebraic variety. \end{para}

\begin{para}The arguments are naturally phrased in terms of the \emph{Borel--Moore homology} groups $\myol H_i(X,\Z)$, which can be defined as the reduced homology groups of the one-point compactification of $X$. If $X$ is smooth then Poincar\'e--Lefschetz duality gives an isomorphism $\myol H_i(X,\Z) \cong H^{2d-i}(X,\Z)$ where $d=\dim_\C(X)$.\end{para}

\begin{prop}\label{cellprop1}Let $X$ be a complex algebraic variety with an algebraic cell decomposition. Then the odd degree Borel--Moore homology groups of $X$ vanish, and for all $i$ there is an isomorphism $\myol H_{2i}(X,\Z) \cong \Z^{ \oplus n_i}$, where $n_i$ denotes the number of cells of dimension $i$. 	In particular, if $X$ is in addition smooth and of pure dimension $d$ then we also have that the odd cohomology groups of $X$ vanish, and $H^{2i}(X,\Z) \cong \Z^{\oplus n_{d-i}}$, by Poincar\'e duality.\end{prop}

\begin{proof}
	This is proven by induction on the number of cells, the base case being that $X = \varnothing$. For the induction step we consider the filtration
	$$ \varnothing = U^{(0)} \subset U^{(1)} \subset \ldots \subset U^{(t)} = X$$
of \cref{cell def}, with $U^{(t)} \setminus U^{(t-1)} \cong \A^d$. Then we know the result for $U^{(t-1)}$ by induction. To conclude the result for $X=U^{(t)}$ we use the long exact sequence in Borel--Moore homology
	\[   \ldots \to \myol H_i(U^{(t-1)},\Z) \to \myol H_i(X,\Z_\ell) \to \myol H_i(\C^d,\Z) \to   \myol H_{i-1}(U^{(t-1)},\Z) \to \ldots  \]
	
	and the fact that $\myol H_i(\C^{d},\Z) = 0$ for $i \neq 2d$, $\myol H_{2d}(\C^{d},\Z) \cong \Z$. 
\end{proof}

\begin{rem}
	A common variant of \cref{cellprop1} in the literature is to assume $X$ proper but not necessarily smooth. For proper varieties Borel--Moore homology coincides with ordinary homology, and under the hypotheses of \cref{cellprop1} one has  $H^{2i}(X,\Z) \cong \Z^{\oplus n_{i}}$. If $X$ is neither proper nor smooth there is no reason that one should be able to determine its cohomology groups from an algebraic cell decomposition. 
\end{rem}

\begin{para}
In generalizing \cref{cellprop1} to the non-Hausdorff setting one runs into the problem that Borel--Moore homology (or its dual, cohomology with compact support), and the Verdier duality theory, are typically only defined for Hausdorff spaces. (But see \cite{moerdijk-mrcun} for a basic theory of compactly supported sheaf cohomology for non-Hausdorff manifolds.) In this respect the situation in \'etale cohomology is in fact significantly nicer. Although textbook treatments of \'etale cohomology only define the functors $Rf_!, f^!$ for separated morphisms, this can now be completely sidestepped using the work of Laszlo--Olsson \cite{laszlo-olsson1,laszlo-olsson2}. Their main focus is to construct an $\ell$-adic formalism of ``six functors'' for stacks, but even for schemes their lack of separation hypotheses is new. Their approach to the functors $Rf_!$ and $f^!$ is to first construct the Verdier duality functor $\mathbb D$ and then define 
$$ Rf_! = \mathbb D \circ Rf_\ast \circ \mathbb D, \qquad f^! = \mathbb D \circ f^\ast \circ \mathbb D, $$
and this, crucially, does not require $f$ to be separated. 
\end{para}
\begin{para}Let us recall how to define Borel--Moore homology in terms of Grothendieck's six functors. For $f \colon X \to \mathrm{Spec}(k)$ a finite type scheme, say over a separably closed field, its \'etale Borel--Moore homology is defined as $\myol H_i^{\text{\emph{\'et}}}(X,\Z_\ell) = \mathcal H^{-i}(Rf_\ast f^! \Z_\ell)$; compare this with the usual \'etale cohomology $H^i_{\text{\emph{\'et}}}(X,\Z_\ell) = \mathcal H^{i}(Rf_\ast f^\ast \Z_\ell)$. This theory has the following properties: 
\end{para}
\begin{enumerate}
	\item If $X$ is smooth and of pure dimension $d$ then there is a Poincar\'e duality isomorphism
	$$ H^i_{\text{\emph{\'et}}}(X,\Z_\ell) \otimes \Z_\ell(d) \cong \myol H_{2d-i}^{\text{\emph{\'et}}}(X,\Z_\ell) .$$
	\item If $U \subset X$ is open with closed complement $Z$, then there is a long exact sequence
	$$  \ldots \to \myol H_i^{\text{\emph{\'et}}}(Z,\Z_\ell) \to \myol H_i^{\text{\emph{\'et}}}(X,\Z_\ell) \to \myol H_i^{\text{\emph{\'et}}}(U,\Z_\ell) \to   \myol H_{i-1}^{\text{\emph{\'et}}}(Z,\Z_\ell) \to \ldots $$
\end{enumerate}
The first property is immediate from our definition of Borel--Moore homology and the fact that $f^!(-) \cong f^\ast(-) \otimes \Z_\ell(d)[2d]$ if $f$ is smooth of relative dimension $d$ \cite[Lemma 9.1.2]{laszlo-olsson2}. The second property can be proven by considering the distinguished triangle
$$  i_\ast i^! \mathcal F \to \mathcal F \to Rj_\ast j^! \mathcal F \stackrel{+1}{\longrightarrow}$$
where $j \colon U \to X$ and $i \colon Z \to X$ are the inclusions, setting $\mathcal F = f^! \Z_\ell$ where $f \colon X \to \mathrm{Spec}(k)$ is the structural morphism, and applying $Rf_\ast$. The distinguished triangle is constructed in case of finite coefficients in \cite{laszlo-olsson1} as Equation (4.10.ii) and the same reasoning works in the $\ell$-adic case. 
\end{para}
\begin{prop}\label{cellprop}
Let $X$ be a finite type scheme over a base field $k$ equipped with an algebraic cell decomposition. Then the odd degree Borel--Moore homology groups of $X$ vanish, and for all $i$ there is an isomorphism $\myol H_{2i}^{\text{{\'et}}}(X,\Z_\ell) \cong \Z_\ell(i)^{ \oplus n_i}$, where $n_i$ denotes the number of cells of dimension $i$. 	In particular, if $X$ is in addition smooth and of pure dimension $d$ then we also have that the odd cohomology groups of $X$ vanish, and $H^{2i}_{\text{{\'et}}}(X,Z_\ell) \cong \Z_\ell(-i)^{\oplus n_{d-i}}$, by Poincar\'e duality.
\end{prop}

\begin{proof}
This can now be proven in an identical manner as \cref{cellprop1}. 
\end{proof}

\begin{rem}A fortiori we also see that the analogue of \cref{cellprop1} remains valid for the singular cohomology of non-separated complex schemes, by applying the $\ell$-adic \cref{cellprop} and Artin's comparison isomorphism.  \end{rem}

\bibliographystyle{alpha}
\bibliography{database}

\end{document}